\documentclass[11pt,reqno]{amsart}
\usepackage{graphicx,amscd}
\usepackage{amsfonts,amsmath,amssymb}
\newcommand{\e}{\varepsilon}
\newcommand{\supp}{{\mathrm{supp}}}

\newcommand{\dist}{{\mathrm{dist}}}

\newcommand{\C}{\mathbb C}

\renewcommand {\supp}{{\rm supp\,}}
\newtheorem{theorem}{Theorem}[section]
\newtheorem{lemma}[theorem]{Lemma}
\newtheorem{prop}[theorem]{Proposition}
\newtheorem{cor}[theorem]{Corollary}

\theoremstyle{definition}
\newtheorem{defn}[theorem]{Definition}

\newtheorem{example}[theorem]{Example}

\title{On Rational Convexity of Totally Real Sets}
\author{Blake J. Boudreaux}
\address[Blake J. Boudreaux]{The University of Western Ontario}
\email{bboudre7@uwo.ca}

\author{Rasul Shafikov}
\address[Rasul Shafikov]{The University of Western Ontario}
\email{shafikov@uwo.ca}

\begin{document}
\begin{abstract}
	Under a mild technical assumption, we prove a necessary and sufficient condition for a totally real compact set in $\mathbb C^n$ to be rationally convex. This generalizes a classical result of Duval--Sibony.
\end{abstract}

\maketitle

\section{Introduction}
It is an important and generally difficult problem in complex analysis to characterize convexity (polynomial, rational, etc.) of compact sets in complex Euclidean spaces. Quite often such a characterization involves ideas from an area of mathematics not directly related to the definition of convexity. In this paper we are concerned with rational convexity of compact sets in $\mathbb C^n$, see Section~\ref{s.rc} for basic definitions. Our principal result is the following.

\begin{theorem}\label{t.BS}
	Let $S$ be a compact regular totally real set in $\mathbb{C}^n$, $n>1$. Then the following are equivalent:
\begin{enumerate}
	\item[(i)] $S$ is rationally convex.
	\item[(ii)] There exists a smooth strictly plurisubharmonic function $\varphi:\mathbb{C}^n\to\mathbb{R}$ and a finite regular cover $\{\Sigma_{k_j}\}_{j=1,\ldots ,r}$ of $S$ by totally real manifolds  that are isotropic (Lagrangian) with respect to the K\"ahler form $\omega = dd^c\varphi$,  i.e., that satisfy
$\iota_{\Sigma_{k_j}}^* \omega=0$ for each $j=1,\ldots,r$, where $\iota_{\Sigma_{k_j}}: \Sigma_{k_j}\to\mathbb{C}^n$ denotes the inclusion map.
\end{enumerate}
\end{theorem}

This theorem is a generalization of a well-known result of Duval--Sibony~\cite{DS} in which $S$ was assumed to be a compact smooth totally real manifold. A compact $S\subset \mathbb C^n$ is said to be a {\it
totally real set} if it is the zero locus of a nonnegative strictly plurisubharmonic function defined in a neighbourhood of $S$. In particular, every totally real manifold is a totally real set. The study of totally real sets was pioneered by Wells~\cite{We} and Harvey--Wells~\cite{HW0,HW1}. A priori, totally real sets may have no  regularity, but it known that locally they are contained in totally real manifolds. This gives a cover of  a totally real set $S$ by totally real manifolds, in general of different dimension. We call the cover regular if any intersection of the manifolds in the cover is also a manifold. This property allows us to construct a special cover of $S$ by totally real manifolds $\{\Sigma_{k_j}\}_{j=1,\ldots ,r}$. We call a totally real set {\it regular} if it admits a regular cover. A prominent feature of this cover is that its closure is itself a totally real set and a stratified space satisfying Whitney condition (B), this is the content of Section~\ref{s.trs} below.

The proof of the theorem, which given in Section 4, in fact, gives a slightly more refined statement, which is also new in the case of totally real manifolds.

\begin{cor}\label{c.global}
Let $M\subset\mathbb{C}^n$, $n>1$, be a regular totally real set. A compact $S\subseteq M$ is rationally convex if and only if there exists a smooth strictly plurisubharmonic function $\varphi:\mathbb{C}^n\to\mathbb{R}$, a neighbourhood $U$ of $S$, and a regular cover $\{\Sigma_{k_j}\}_{j=1,\ldots,r}$ of $M$ such that $\iota^*_{U\cap\Sigma_{k_j}}dd^c\varphi=0$ for each $j=1,\ldots,r$. In particular,  if $M$ is a totally real manifold, then $M$ can be taken to be the cover, i.e., $S$ is rationally convex iff  $\iota^*_{U\cap M} dd^c\varphi = 0$.
\end{cor}

The proof of Corollary \ref{c.global} will be given in Section 4.

\smallskip

Some generalizations of the Duval--Sibony theorem were also obtained by Gayet~\cite{Ga}, Duval--Gayet~\cite{DG1}, Shafikov--Sukhov~\cite{SS1}, and Mitrea~\cite{Mi}. In these results $S$ is either an immersed manifold or has special isolated singularities. In the case $\dim S = n$, such an $S$ cannot be a totally real set, and so these results apply to a different class of compacts. On the other hand, the set $S$ in our result need not have any regularity at all, in particular, it may have a fractal-type behaviour.

As an application, combining Theorem~\ref{t.BS} with the work of Berndtsson~\cite{Be}
we obtain the following approximation result.

\begin{cor}\label{c.approx}
Suppose that $S \subset \mathbb C^n$, $n>1$, is a compact totally real set with a regular cover $\{\Sigma_{k_j}\}_{j=1,\ldots ,r}$ that is isotropic (Lagrangian) with respect to some K\"ahler form on $\mathbb C^n$. Then any complex-valued continuous function on $S$ can be approximated uniformly on $S$ by rational functions with the poles off $S$.
\end{cor}

The proof of Corollary~\ref{c.approx} is immediate from Theorem~\ref{t.BS}; see Section 2.

\smallskip

The condition that $S$ admits a regular cover is a technical assumption needed for our arguments to work. We do not know if this assumption is really necessary or if there exist totally real sets that do not admit a regular cover.

\medskip

\noindent{\bf Acknowledgments.} We would like to thank Matthias Franz for helpful discussions concerning ANR spaces. We would also like to thank the anonymous referee for valuable comments, in particular for suggesting the statement of Proposition~\ref{p.refe}. The second author is partially supported by Natural Sciences and Engineering Research Council of Canada.

\section{Rational Convexity}\label{s.rc}
Given a  compact set $S \subset \mathbb C^n$, the {\it rationally convex hull} of $S$, denoted by $R(S)$, is defined as
\[
R(S) = \{ z \in \mathbb C^n : |P(z)| \le || P ||_{S},  P \text{ is any rational function with poles off } S \}.
\]
We say that $S$ is rationally convex if $R(S) = S$. This is equivalent to the following: for any point $z_0\in \mathbb C^n \setminus S$ there exists a holomorphic polynomial $P(z)$ on $\mathbb C^n$ such that $P(z_0)=0$ but $P$ does not vanish on $S$. Any compact in $\mathbb C$ is rationally convex, but in higher dimensions it is generally difficult to determine whether a given compact is rationally convex or not, see, e.g., Stolzenberg~\cite{Stol} for an early work in this direction.

A general sufficient condition for rational convexity is given by Duval--Sibony~\cite[Theorem 1.1]{DS}: {\it if $\varphi$ is a plurisubharmonic function on $\mathbb C^n$, $\mathbb C^n \setminus \supp dd^c \varphi$ is compact, then for any $s>0$, the set
$$
K_s = \{z\in \mathbb C^n : \dist (z, \supp dd^c \varphi )\ge s\}
$$
is rationally convex.} This theorem implies the following characterization of rational convexity of totally real sets.

\begin{prop}\label{p.refe}
	Let $S\subset\mathbb{C}^n$ be compact totally real set given as the zero locus of a nonnegative strictly plurisubharmonic function $\varphi$ defined in a neighbourhood of $S$. Then
$S$ is rationally convex if and only if $dd^c\varphi$ extends to a K\"ahler form on $\mathbb C^n$.
\end{prop}

The ``if" direction of Proposition~\ref{p.refe}  follows immediately from the proposition in Nemirovski~\cite{N}, which is
a corollary of Theorem~1.1 of Duval--Sibony stated above. The proof in the other direction  can be deduced from the content of Section~\ref{p.easy}.

Further, Duval--Sibony theorem~\cite[Theorem 3.1]{DS} states the following: {\it a smooth compact totally real manifold in~$\mathbb C^n$ is rationally convex if and only if it is isotropic with respect to a K\"ahler form in~$\mathbb C^n$}.
Our main result is a generalization of this theorem.

Finally, Duval--Sibony~\cite[Theorem 2.1]{DS} gave the following characterization of the rationally convex hull of an arbitrary compact set $S\subset \mathbb C^n$: {\it for every $z\notin R(S)$ there exists a smooth positive closed
$(1, 1)$-form $\omega$ that is strictly positive at $z$ and vanishes in a neighbourhood of $R(S)$.
} This gives a way to construct the K\"ahler form with respect to which a totally real rationally convex
manifold $S$ is isotropic. In the other direction the proof of Duval--Sibony  relies on the following result
(cf.~\cite[Lemma 1.2]{DS}):
{\it Suppose $\varphi$ is a continuous plurisubharmonic function on $\mathbb C^n$, and $h$ is a holomorphic function on some domain $V\subset \mathbb C^n$. Assume that
\[
K = \{ z\in V:  |h(z)| \ge e^{\varphi(z)}\}
\]
is compact. Then for every $z\in \mathbb C^n \setminus K$ there exists an entire holomorphic function $f$ such that $f(z)=0$ but the hypersurface $\{f=0\}$ omits $K$.
}Algebraic approximation of $f$ then shows that the set $K$ is rationally convex, and so the proof of rational convexity of $S$ boils down to the construction of the required functions $\varphi$ and $h$ so that $S=K$. These ideas will be used in the proof of our main result.

Rational convexity is important, in particular, in view of the Oka--Weil theorem, see, e.g., Stout~\cite{St}.
It states that if $S$ is a rationally convex compact, then any function holomorphic on $S$ can be approximated uniformly on $S$ by rational functions with poles off $S$. By Berndtsson~\cite{Be}, any continuous function on $S$ can be approximated by functions holomorphic
on $S$. This gives the proof of Corollary~\ref{c.approx}: by Theorem~\ref{t.BS} the set $S$ is rationally convex, and so combining the Oka--Weil theorem with Berndtsson we obtain the required approximation.

Conversely, if $S$ is a compact such that any continuous function on $S$ can be approximated uniformly on $S$ by rational functions with poles off $S$, then $S$ is rationally convex, see Stout~\cite[Theorem 1.2.10]{St}. Combined with Corollary~\ref{c.approx} this implies that any compact subset of a rationally convex totally real set  $S$ is itself rationally convex. This gives the proof of Corollary~\ref{c.global} in one direction; see Section~\ref{p.cor} for a complete proof.
Note that this simple argument does not imply Theorem~\ref{t.BS} because a totally real set is not necessarily contained in a totally real manifold, as we will see in the next section.

\section{Totally real sets}\label{s.trs}
Recall that a smooth manifold $M$ is totally real, if for any point $p\in M$, the tangent plane $T_pM$ does not contain any complex directions. A generalization of this is the notion of a totally real set.
In this section we give a quick introduction to this subject and then we define a special subclass of totally real sets that we call {\it regular}. Totally real sets can be defined on arbitrary complex manifolds, but for simplicity we restrict our attention to compacts in~$\mathbb C^n$, a general reference to totally real sets is Stout~\cite{St}.

\begin{defn}\label{d.TRS}
A compact subset $S \subset \mathbb C^n$ is called a {\it totally real set} if there exist a neighbourhood $U$ of $S$ in $\mathbb C^n$ and a nonnegative strictly plurisubharmonic function $\varphi(z)$ defined on $U$ such that $S=\{\varphi=0\}$.
\end{defn}

In fact, being a totally real set is a local notion.  More precisely, the following holds: {\it if for every point $p$ in a compact set $S$ there exist a neighbourhood $U_p$ of $p$ in $\mathbb C^n$ and a nonnegative smooth strictly plurisubharmonic function $\varphi_p$ such that $S \cap U_p = \{z \in U_p : \varphi_p (z) =0\}$, then $S$ is a totally real set.} This can be proved by considering a locally finite cover of $S$ by open sets $U_p$ and using a partition of unity argument (see Lemma 6.1.3 of Stout \cite{St}).
In particular, it follows that any compact totally real submanifold $M$ of $\mathbb C^n$ is a totally real set. Indeed, it is well-known that locally the square-distance function to $M$ is strictly plurisubharmonic thus giving the required function $\varphi_p$ near every point $p\in M$.

The following result was proved by Harvey--Wells~\cite{HW1}:
{\it Let $\varphi$ be a nonnegative  strictly plurisubharmonic function of class $C^{k+1}$, $k\ge 1$, on the open  set $U \subset \mathbb C^n$, and let $S$ be its zero locus. For every $p\in S$ there exists a neighbourhood  $U_p$ of $p$ and a totally real manifold $M_p$ of class $C^k$ in $U_p$ such that
$S \cap U_p \subset M_p$.} (In this paper we avoid questions of minimal finite smoothness required for the arguments to go through and simply assume $C^\infty$-smoothness of all the objects involved.)
On the other hand, suppose that $S\subset M$ is a compact subset of a totally real submanifold $M\subset \mathbb C^n$. Then $S$ can be represented as the zero set of a smooth function $h$ (see Lemma 1.4.13 of Narasimhan \cite{Na}). By taking $\varphi$ to be the square-distance function to $M$ and a sufficiently large even integer $m$ we see that the function $\varphi(z) + h^m(z)$ is strictly plurisubharmonic in a small neighbourhood of $M$ and its  zero locus is exactly $S$. This shows that {\it a compact $S$ is a totally real set if and only if $S$ is locally contained in a totally real manifold}.

A natural question is whether any totally real set is globally contained in some smooth totally real manifold. A positive answer to this question would of course undermine the importance of this class of compacts. This is however not the case. Consider the following example due to Chaumat--Chollet~\cite{ChCh}.

\begin{example}\label{e.CC}
Consider the map $F:\mathbb R^3 \to \C^3$ given by
\[
F (t_1, t_2, t_3) = \left(t_1 \cos t_3, t_1 \sin t_3,  t_2\, e^{i t_3/2}\right).
\]
One can verify that this is a totally real immersion of $\mathbb R^3 \setminus \{t_1=0\}$ into
$\mathbb C^3$. The restriction of $F$ to a subdomain
\[
D = \{ t\in \mathbb R^3 : (t_1,t_2,t_3) \in (0,2) \times (-1,1) \times \mathbb R\}
\]
can be seen as the universal cover of its image $\Sigma:=F(D)$, which is a totally real submanifold of $\mathbb C^3$ of dimension 3. One can see that $D$ contains the infinite strip
\[
T=\{t\in \mathbb R^3 : t_1=1, -1/2<t_2<1/2, t_3\in\mathbb R\},
\]
which is mapped by $F$ onto a compact subset of $\Sigma$. The set $M=F(T)$ is a M\"obius strip.
Consider the set $S = M \cup \Delta$, where $\Delta$ is a disc
\[
\Delta = \{z\in \mathbb C^3 : x_1^2 + x_2^2 \le 1, \ y_1=y_2=z_3 =0\}.
\]
Then
\[
\Gamma = M \cap \Delta = \{z\in \mathbb C^3: x_1^2+x_2^2 =1, y_1=y_2=z_3=0\}.
\]
One can verify that $S\setminus\{0\}\subset \Sigma$ and that $\Delta$ is contained in $\mathbb R^3_x$. Hence, $S$ is a totally real set. However, $S$ is not contained in a totally real submanifold of an open set in $\mathbb C^3$. Indeed, suppose $\Sigma_0$ is a totally real 3-dimensional manifold that contains $S$. The tangent bundle $T\Sigma_0$ when restricted to $\Delta$ is trivial, since $\Delta$ is contractible, so in particular, $T\Sigma_0|_{\Gamma}$ is trivial. For each $p\in\Gamma$, $T_p \Sigma_0$ contains both $T_p M$ and $T_p \mathbb R^2_{(x_1,x_2)}$. But this is impossible because no neighbourhood of $\Gamma$ in $M$ is orientable.  $\diamond$
\end{example}

Let $S\subset \mathbb C^n$ be a totally real set. By the discussion above there exists a locally finite cover $\{U_\alpha\}_{\alpha\in A}$ of $S$ by open sets in $\mathbb C^n$ such that for each $\alpha\in A$, the set $S\cap U_\alpha$ is contained in some totally real submanifold $M_\alpha$ of $U_\alpha$. Then $\{M_\alpha\}_{\alpha\in A}$ is a cover of $S$ by totally real manifolds. Note that by compactness of $S$ the set $A$ can always be chosen to be finite.

\begin{defn}
We say that $\{M_\alpha\}_{\alpha\in A}$ is a {\it regular cover} of $S$ if any intersection of manifolds in $\{M_\alpha\}$ is either empty or is a smooth manifold. We say that a totally real set $S$ is {\it regular}, if it admits a regular cover $\{M_\alpha\}$.
\end{defn}

As an example, consider a compact real-analytic set $S\subset \mathbb C^n$ of dimension at most $n-1$ which has only isolated singularities and which is a totally real set (in the sense of Definition~\ref{d.TRS}). Then $S$ is a regular totally real set. Indeed, the regular part of $S$ is a disjoint union of totally real manifolds (in general not of the same dimension even if $S$ is irreducible), denote these by $M_1, M_2, \dots, M_k$. Let $q_1, q_2,\dots, q_m$ be the singular points of $S$. For each $q_j$ there exists a local totally real manifold $M_{q_j}$ that contains a neighbourhood of $q_j$ in $S$. We may choose these manifolds to be disjoint. Then $\{M_1, \dots, M_k, M_{q_1},\dots, M_{q_m}\}$ is a regular cover of $S$, since $M_j \cap M_{q_l}$ is either empty or is an open subset of $M_j$.
A detailed discussion of totally real analytic sets can be found in Wells~\cite{We}. Another instance of a regular totally real set is given in Example~\ref{e.CC}: the manifolds $\Sigma$ and $\Delta$ form a regular cover of the set $S$. Thus, a regular totally real set in general is not contained in a totally real manifold. Finally, we note that we do not have any examples of totally real sets that are not regular.

A priori, a totally real set may have no regularity, as any compact subset of a totally real manifold is a totally real set. This is a general difficulty when working with such objects. As a way to overcome this problem we now construct a new totally real set that contains the given regular totally real set and is a finite union of smooth manifolds with the special intersection property described above. This construction will be used in the proof of the main theorem. For the last statement in the proposition below recall that  a closed subset ${X}$ of a topological space $Y$ is called a {\it neighbourhood retract} of $Y$ if $X$ is a retract of some open subset of $Y$ that contains $X$; the subset $X$ is called an {\it absolute neighbourhood retract} or ANR
(for the class of metrizable topological spaces), if $X$ is a neighbourhood retract of $Y$ whenever $X$ is a closed subset of a metric space $Y$. For further details see, e.g., Fritsch--Piccinini~\cite{FP}.

\begin{prop}\label{p.BS}
Let $S\subset \mathbb C^n$ be a regular totally real set and let $\delta>0$ be arbitrary. Then there exists a collection of smooth totally real manifolds $\Sigma_{k_j}$, $j=1,\dots, r$,
$\dim \Sigma_{k_j}=k_j$, $k_1<k_2< \dots < k_r$, with the following  properties

\begin{itemize}
\item[(i)] $S\subset \Sigma :=\bigcup_{1\le j \le r} \Sigma_{k_j}$;

\item[(ii)] if $p \in \Sigma_{k_j} \cap \Sigma_{k_l}$, $k_j < k_l$, then there exists a neighbourhood $V_p\subset \mathbb C^n$ of $p$ such that  $V_p \cap \Sigma_{k_j} \subset \Sigma_{k_l}$. In particular, $\{\Sigma_{k_j}\}_{j=1,\dots,r}$ is a regular cover of $S$;

\item[(iii)] $\overline{\Sigma}$ is a totally real set;

\item[(iv)] $\overline \Sigma$ is contained in the $\delta$-neighbourhood of $S$.

\item[(v)] The compact $\overline\Sigma$ can be stratified to satisfy Whitney condition (B). It is also ANR, in particular, there exists a neighbourhood basis of $\overline\Sigma$ that retracts to $\overline\Sigma$.

\end{itemize}
\end{prop}

Throughout the paper $B(z,\rho)$ denotes the Euclidean ball in $\mathbb C^n$ of radius $\rho$ centred at $z$.

\begin{proof}
Let $\mathcal M = \{M_\alpha\}$ be a finite regular cover of $S$. For any $p\in S$ consider the manifold $M_p = \cap_{p\in M_\alpha} M_\alpha$, i.e., the intersection of all manifolds in the cover that contain $p$. From the definition of a regular cover, in some neighbourhood $U_p$ of $p$, $M_p$ is a totally real submanifold  of $U_p$ of some dimension $k_p$ such that $S\cap U_p \subset M_p$. Note that if $k_p=0$ then $p$ is an isolated point of $S$. We call $M_p$ the canonical manifold at $p$ with respect to $\mathcal M$, and $k_p$ the index of $p$. Let $0 \le k_1 < k_2 < \dots < k_r \le n$ be the list of all indices appearing in $S$. Our proof is a reverse induction on $k_j$, $j=1,\dots,r$. Define
\[
S_{k_j} = \{p\in S : \text{index}(p) = k_j\}.
\]
We claim that the set $\tilde S = \cup_{j\le l} S_{k_j}$ is an open subset of $S$ (in the topology induced by $\mathbb C^n$) for all $l$, and hence, $S\setminus \tilde S$ is a closed subset of $\mathbb C^n$. Indeed, if $p\in \tilde S$, then in a neighbourhood $V$ of $p$, the set $S\cap V$ is contained in a manifold of dimension $k_j$ for some $j \le l$, and therefore,
$\text{index}(q) \le k_j$ for all $q \in V_p \cap S$. Therefore, $p$ belongs to $\tilde S$ together with its small neighbourhood, which shows that $\tilde S$ is open, and $S\setminus \tilde S$ is closed.

Consider $S_{k_r}$, the closed subset of $S$ containing points of the top index. For every $p\in S_{k_r}$ there exists
$\e = \e(p)>0$ such that the canonical manifold $M_p$ is a submanifold of $B(p, 2\e)$, and $S \cap B(p, 2\e) \subset M_p$.  Set
\[
M(p,\e)=M_p \cap B(p, \e).
\]
Then $S_{k_r}$ admits a finite cover by manifolds $M(p_j,\e)$, $p_j\in S_{k_r}$. We now show that after some surgery on the manifolds $M(p_j,\e)$, they can be glued together to form one manifold. Indeed, suppose $M(p_1,\e) \cap M(p_2,\e) =K\ne \varnothing$.
From the properties of the regular cover $\mathcal M$, the set $K$ is a submanifold of $B(p_1, \e) \cap B(p_2,\e)$.
If $\dim K < k_r$, then $K\cap S_{k_r} = \varnothing$, as otherwise, $S_{k_r}$ would contain points of index $\le \dim K$.
Further, since the same property also holds in the balls of radius $2\e$, we conclude that $\overline K \cap S_{k_r} = \varnothing$. Therefore, we may remove from $M(p_1,\e)$ and $M(p_2,\e)$ a small closed neighbourhood of $\overline K$ without affecting their intersection with $S_{k_r}$. The only other possibility is that $\dim K = k_r$, which means that the manifolds $M(p_1,\e)$ and $M(p_2,\e)$ agree on the intersection and can be glued together. Repeating this procedure for all $M(p_j,\e)$, we conclude that the manifolds $M(p_j,\e)$ can be glued together to form a manifold $\Sigma_{k_r}$ of dimension $k_r$ that contains~$S_{k_r}$. The manifold $\Sigma_{k_r}$ has only finitely many connected components.

We now continue by induction. Suppose that for $m>1$ we have constructed manifolds
$\Sigma_{k_m}, \Sigma_{k_{m+1}}\dots, \Sigma_{k_r}$, $\dim  \Sigma_{k_j} = k_j$, that satisfy the intersection property (ii) of the proposition and such that
\[
\bigcup_{m\le j\le r} S_{k_j} \subset \bigcup_{m\le j\le r} \Sigma_{k_j} .
\]
We outline the construction of the manifold $\Sigma_{k_{m-1}}$. Let
\begin{equation}\label{e.R}
R_m = \cup_{j\ge m}\Sigma_{k_j}.
\end{equation}

\begin{lemma}\label{l.1}
$S_{k_{m-1}} \setminus R_m$ is a closed subset of $S$.
\end{lemma}

\begin{proof}
Indeed, if $(p_\nu)$ is a sequence in $S_{k_{m-1}} \setminus R_m$ that converges to some point $p_0$, then $p_0 \in \cup_{j\ge m-1} S_{k_j}$ , since $\cup_{j\ge m-1} S_{k_j}$ is a closed set as shown above. On the other hand, $p_0$ cannot be a point in $R_m$ as otherwise it would be contained in $R_m$ together with a small neighbourhood. And since $R_m$ contains all points in $\cup_{j\ge m}S_{k_j}$,
we conclude that $p_0 \in S_{k_{m-1}} \setminus R_m$. This shows that
$S_{k_{m-1}} \setminus R_m$ is closed.
\end{proof}

 As in the case of $S_{k_r}$ discussed above, for every point $p\in S_{k_{m-1}} \setminus R_m$ there exists $\e=\e(p)$ such that the canonical manifold $M_p$ is a submanifold of $B(p, 2\e)$, and $S \cap B(p, 2\e) \subset M_p$. Then the set $S_{k_{m-1}} \setminus R_m$ can be covered by a finite collection of manifolds $M(p_j,\e)=M_{p_j} \cap B(p_j, \e)$,
$p_j \in S_{k_{m-1}} \setminus R_m$. As above, if the intersection of two such manifolds is nonempty, then either a small neighbourhood of the intersection can be removed from $M(p_j, \e)$ without affecting their intersection with $S_{k_{m-1}}$, or the manifolds coincide near the intersection. After removing all lower-dimensional intersections, $M(p_j, \e)$ can be glued together to form a
${k_{m-1}}$-dimensional manifold $\Sigma_{k_{m-1}}$ that contains $S_{k_{m-1}} \setminus R_m$. Suppose now that
$K = \Sigma_{k_{m-1}} \cap \Sigma_{k_{j}}$ for some $j>m-1$, $K\ne \varnothing$. Note that $K$ is constructed as an intersection of some submanifolds in $\mathcal M$, and therefore, by the regularity of $\mathcal M$, it is a smooth manifold. Suppose that $\dim K < k_{m-1}$, and let $p\in \overline K \cap \Sigma_{k_{m-1}} \cap S$. Then $p$ belongs to the closure of some manifold $M(q,\e)$ of dimension $k_j$ that was used in the construction of $\Sigma_{k_{j}}$. By increasing this $\e$ slightly we see that near $p$  the set $S$ is locally contained in a totally real manifold of dimension $=\dim K$, i.e., the index of $p$ is less than $k_{m-1}$. Since the points of index less than $k_{m-1}$ form an open set  in $S$, there exists a closed neighbourhood of $K$ that is disjoint from $S_{k_{m-1}}$. By removing this neighbourhood from $\Sigma_{k_{m-1}}$ we can ensure that the latter does not intersect $\Sigma_{k_{j}}$ along a manifold of lower dimension. And if $\dim K = k_{m-1}$, then it simply means that $K$ is an open subset of
$\Sigma_{k_{m-1}}$. A similar analysis holds when $p\in \overline K \cap \Sigma_{k_{j}} \cap S$ or
when $p$ is a point in the intersection of the boundaries of $\Sigma_{k_{j}}$ and $\Sigma_{k_{m-1}}$.
This gives the manifold $\Sigma_{k_{m-1}}$ that satisfies (ii).

This inductive procedure gives the required manifolds of all dimensions $k_1, \dots, k_r$. Note that if $k_1 = 0$, then $S_{k_1}$ consists of isolated points, these are open sets in $S$. Further, the set
$S \setminus \bigcup_{1< j\le r} \Sigma_{k_j}$ consists of finitely many such points, and their union is $\Sigma_{k_1}$.  The cover $\{\Sigma_{k_j}\}$ is regular because intersection of any of the manifolds in the cover is an open subset of the manifold of the smallest dimension in the intersection. This verifies properties (i) and~(ii).

By construction, every point in the closure of $\Sigma_{k_j}$ belongs to a totally real manifold. This implies (iii). Finally, property (iv) can be achieved by choosing all $\e$ involved in the construction of $\Sigma_{k_j}$ to be less than the given $\delta$.

\medskip

Proof of (v). To satisfy this property we will need to further modify the set
\begin{equation}\label{e.clS}
\overline \Sigma =\overline{\bigcup_{1\le j \le r} \Sigma_{k_j}} .
\end{equation}
Returning to the inductive construction of $\Sigma$, note that the set $S_{k_r}$ is compactly contained
in~$\Sigma_{k_r}$. Therefore, after a small shrinking followed by a small perturbation of the boundary
of $\Sigma_{k_r}$ we may assume that $\overline{\Sigma_{k_r}}$ is a manifold with boundary, in particular, the boundary $b\Sigma_{k_r}$ is a smooth closed manifold of dimension $k_r -1$. Similarly, by Lemma~\ref{l.1} for all $1 < m \le r$, the set $S_{k_{m-1}}\setminus R_m$, where $R_m$ is defined by~\eqref{e.R}, is compactly contained in
$\Sigma_{k_{m-1}}$. Again, after a small shrinking and perturbation we may assume that $\overline{\Sigma_{k_{m-1}}}$ is a
manifold with boundary that compactly contains $S_{k_{m-1}}\setminus R_m$. We conclude that
$\Sigma_{k_j}$, $1 \le j \le r$, can be chosen so that $\overline{\Sigma_{k_j}}$ are manifolds with boundary and properties (i)-(iv) still hold. By construction, each $\overline{\Sigma_{k_j}}$ is compactly contained in a bigger manifold of the same dimension, which we denote by $\tilde \Sigma_{k_j}$ (the original $\Sigma_{k_j}$). Also note that $b\Sigma_{k_j}$ is a closed submanifold of $\tilde \Sigma_{k_j}$.

In the manifold $\tilde \Sigma_{k_r}$ consider the closed submanifold $b\Sigma_{k_r}$ and some
$\Sigma_{k_j}$, $j< r$, $k_j>0$, that has nonempty intersection with $\Sigma_{k_r}$. By construction of
$\Sigma_{k_j}$ and from property (ii) of the proposition, we may assume that near $b\Sigma_{k_r}$
we have the inclusion $\tilde\Sigma_{k_j} \subset \tilde \Sigma_{k_r}$. Taking $\tilde \Sigma_{k_r}$ as the ambient space, we may apply Thom's transversality theorem~\cite{GG} to conclude that after a small perturbation of $b\Sigma_{k_r}$, the manifold $b\Sigma_{k_r}$ intersects $\tilde \Sigma_{k_j}$ transversely. After a small perturbation of $b \Sigma_{k_j}$ we may further assume that
$b\Sigma_{k_r}$ and $b \Sigma_{k_j}$ also intersect transversely (note that the latter condition in general does not follow from the transversality of $b\Sigma_{k_r}$ and $\tilde \Sigma_{k_j}$).
Finally, since transversality is stable under small perturbations, we may repeat this procedure for all $j \ne r$ to ensure that all $\tilde \Sigma_{k_j}$ and
$b\Sigma_{k_j}$, $j\ne r$, intersect $b\Sigma_{k_r}$ transversely.

We now continue by induction: once a small perturbation of $b\Sigma_{k_r}$ ensures that intersections
of $b\Sigma_{k_r}$ with manifolds $\Sigma_{k_j}$ and $b\Sigma_{k_j}$ are transverse for all $j<r$, we may continue with perturbation of $b\Sigma_{k_{r-1}}$ within $\tilde \Sigma_{k_{r-1}}$ so that its intersection with
$\Sigma_{k_j}$ and $b\Sigma_{k_j}$ is transverse for all $j < r-1$. Then repeat this for all  $b\Sigma_{k_j}$, $1\le j\le r$. This is possible because the procedure requires a finite number of perturbations and transversality is stable under small perturbations.

To continue with our argument we note the following elementary result whose proof is left for the reader.

\begin{lemma}\label{l.trans}
Let $X$ be a smooth manifold, $Y \subset X$ be a closed submanifold, and $Z \subset X$ be a relatively compact manifold with boundary $bZ$. If $Z$ and $bZ$ intersect $Y$ transversely, then $Y \cap Z$ is a manifold with boundary $bZ \cap Y$.
\end{lemma}

We now give a locally finite stratification of the set $\overline \Sigma$ into smooth manifolds. Basically, it is obtained by taking connected components of all possible intersections and their complements of manifolds $\Sigma_{k_j}$ and $b\Sigma_{k_l}$, $1\le j,l \le r$. This can be formally organized by the following induction procedure. To begin with, consider
\begin{equation}\label{e.strat1}
\overline \Sigma_{k_r} = \Sigma_{k_r} \bigcup
\left(b\Sigma_{k_r} \setminus (\cup_{j=1}^{r-1} \overline{\Sigma_{k_j}})\right) \bigcup W_{r-1},
\end{equation}
where $W_{r-1} = b\Sigma_{k_r} \cap (\cup_{j=1}^{r-1} \overline{\Sigma_{k_j}})$ is a compact set that is contained in $\cup_{j=1}^{r-1} \overline \Sigma_{k_j}$. The connected components of the first two terms in the union in~\eqref{e.strat1} give a stratification of $\overline \Sigma_{k_r} \setminus W_{r-1}$ into disjoint smooth manifolds of dimension $k_r$ and $k_r - 1$.
This can be continued inductively: for $1< m \le r$, let $R_m$ be defined as in~\eqref{e.R}.
We write $\overline{R_m} = T_m \cup W_{m-1}$, where
\[
W_{m-1} = (\overline{(\overline \Sigma \setminus \overline{R_m})}) \cap \overline{R_m} , \ \
T_m = \overline{R_m} \setminus W_{m-1} .
\]
Note that for $m=r$, $\overline{R_r} = \overline{\Sigma_{k_r}}$, and so this decomposition agrees with~\eqref{e.strat1}. For $m<r$ we have
\[
T_m= \bigcup_{j\geq m} \left[\Sigma_{k_j} \cup \left( b\Sigma_{k_j} \setminus (\cup_{l<j} \overline \Sigma_{k_l})\right) \right] .
\]
Assuming that $T_m$ is already stratified, we give a stratification of $T_{m-1}$. Note that $T_m \subset T_{m-1}$, and $(T_{m-1} \setminus T_m) \subset \overline{\Sigma_{k_{m-1}}}$. Consider the following decomposition.
\begin{multline}\label{e.strat2}
\overline{\Sigma_{k_{m-1}}} \setminus T_m =
\Sigma_{k_{m-1}} \setminus (\cup_{j \ge m} \overline{\Sigma_{k_j}})
\bigcup_{j \ge m} (\Sigma_{k_{m-1}} \cap b \Sigma_{k_j}) \\
\bigcup_{j \ge m} \left[(b\Sigma_{k_{m-1}} \cap b \Sigma_{k_j}) \setminus (\cup_{l<m-1} \overline {\Sigma_{k_l}}) \right]
\bigcup \left[ b\Sigma_{k_{m-1}} \setminus (\cup_{j\ne m-1} \overline{\Sigma_{k_j}})\right] \bigcup W_{m-2},
\end{multline}
where again, $W_{m-2}$ is a compact set contained in $\cup_{j=1}^{m-2} \overline \Sigma_{k_j}$, or empty if $m=2$. By Lemma~\ref{l.trans}, all the terms above are smooth manifolds, and so  formula~\eqref{e.strat2} gives stratification of $\overline \Sigma_{k_{m-1}} \setminus T_m$ into connected manifolds of dimension $k_{m-1}$ and $k_{m-2}$.
Repeating this inductive procedure for all $m$ gives the required stratification of $\overline\Sigma$.

Observe that the obtained stratification of $\overline \Sigma$ satisfies the frontier condition, i.e.,
if $X$ and $Y$ are two strata and $Y\subset \overline X$, then $Y\subset \overline X \setminus X$.
We now claim that this stratification satisfies Whitney condition (B) (for a general reference on stratified spaces see, e.g., Trotman~\cite{Tro}). Recall that a stratified space satisfies Whitney condition
(B) if the following holds: Let $X$ and $Y$ be two adjacent strata of the stratification (i.e., $Y \subset \overline X \setminus X$). Suppose that the sequences $(x_j) \subset X$ and $(y_j) \subset Y$ both converge to a point $y\in Y$, the sequence of straight lines $l_j$ passing through points $x_j$ and $y_j$ converges to a line $l_0$, and the sequence of the tangent planes $T_{x_j}X$ converges to a plane $T_0$ as $j\to\infty$. Then $l_0 \subset T_0$. This condition, for example, holds  if $\overline X$ is a manifold with boundary, and $Y \subset bX$, or if $X$ is an open subset of a larger manifold with boundary and $Y$ is a manifold in the topological boundary of $X$.

To see that the stratification of $\overline\Sigma$ constructed above satisfies Whitney condition (B) assume that $X$ and $Y$ are some strata defined by~\eqref{e.strat1} or~\eqref{e.strat2} such that $Y \subset \overline X \setminus X$. Consider several cases, where $1 < m \le r+1$.
\begin{itemize}

\item[(a)] $X = \Sigma_{k_{m-1}} \setminus (\cup_{j \ge m} \overline{\Sigma_{k_j}}) $. In this case,
$X$ is an open subset of $\Sigma_{k_{m-1}}$ and the stratum $Y$ is a submanifold in its boundary. Therefore the Whitney condition (B) holds.

\item[(b)] $X= \Sigma_{k_{m-1}} \cap b \Sigma_{k_j}$ for some $j\ge m$.
Then $X$ is an open subset of the closed manifold $b\Sigma_{k_{j}}$, and $Y$ is contained in its boundary, so
Whitney condition (B) holds.

\item[(c)]  $X= (b\Sigma_{k_{m-1}} \cap b \Sigma_{k_j}) \setminus (\cup_{l<m-1} \overline {\Sigma_{k_l}})$ for some $j \ge m$. In this case $X$ is an open subset of $b\Sigma_{k_{m-1}}$ with $Y$ in its boundary. Again, Whitney condition (B) holds.

\item[(d)]  $X= b\Sigma_{k_{m-1}} \setminus (\cup_{j\ne m-1} \overline{\Sigma_{k_j}})$. This case follows in a similar manner.

\end{itemize}
This shows that the stratification of $\overline \Sigma$ satisfies Whitney condition (B).

Stratified spaces satisfying Whitney condition (B) are triangulable, see, e.g., Goresky~\cite{Go}.  It is well-known
(e.g., Fritsch--Piccinini~\cite[Theorem 3.3.10]{FP}) that any CW-complex, in particular, a triangulable space, is ANR. This immediately implies part (v) of the proposition.
\end{proof}
Let us conclude this section with two remarks regarding Proposition~\ref{p.BS}.
\begin{enumerate}
\item The surgeries performed in the proof introduce only finitely many holes. Together with the observation that each manifold in the construction of $\overline\Sigma$ can be chosen to be contractible, this ensures that the stratified space $\overline\Sigma$ has finitely generated first homology class.
\item A more sophisticated topological argument can be used to prove that $\overline\Sigma$ is, in fact, a neighbourhood deformation retract, but we do not need it for the purpose of this paper.
\end{enumerate}

\section{Proof of Theorem~\ref{t.BS} and Corollary~\ref{c.global}}

Here we provide the proof of Theorem~\ref{t.BS}. Each direction of the proof has its own subsection, with a brief interlude to indicate the proof of Corollary \ref{c.global}.

\subsection{Proof of (i) $\Longrightarrow$ (ii).}\label{p.easy}
Suppose that $S$ is a rationally convex regular totally real set. Using Proposition \ref{p.BS} and a given regular cover $\{M_\alpha\}$ of $S$ we construct a regular cover $\{\Sigma'_{k_j}\}_{j=1,\ldots,r}$ of $S$ satisfying (i)-(v). Since
$\overline\Sigma'=\bigcup_{j}\overline\Sigma'_{k_j}$ is totally real, there exists a neighbourhood $U$ of $\overline\Sigma'$ and a nonnegative strictly plurisubharmonic $\varphi_1\in C^{\infty}(U)$ with $\varphi_1^{-1}(0)=\overline\Sigma'$.

Let $B$ be a ball large enough so that $\overline U\subset B$. For each $z\in\overline{B}\setminus U$, by Duval--Sibony~\cite[Theorem 2.1]{DS} (see Section~\ref{s.rc}), there exists a smooth positive closed $(1,1)$ form $\omega_z$ on $\mathbb{C}^n$ that is strictly positive at $z$ and zero in a neighbourhood of $S$.
Select a finite sequence of such forms $\{\omega_\ell\}$ so that if $w\in\overline{B}\setminus U$, then $\omega_\ell(w)>0$ for some $\ell$.
Set $\omega:=\sum_\ell\omega_\ell$. The form $\omega$ is a smooth positive closed $(1,1)$-form on $\mathbb{C}^n$ that is zero on a closed neighbourhood $\overline V\subset U$ of $S$, and is strictly positive elsewhere on a neighbourhood of $\overline{B}$. By shrinking $\delta>0$ in the statement of Proposition \ref{p.BS}, we may find another regular cover $\{\Sigma_{k_j}\}_{j=1\ldots,r}$ of $S$ satisfying properties (i)-(v) with $\overline \Sigma_{k_j}\subset\Sigma'_{k_j}$ for each $j$ and whose union $\overline\Sigma$ is contained in $V$ as a relatively compact subset. Let $\chi_1:\mathbb{C}^n\to [0,1]$ be a smooth function that is identically zero on $\overline{U}$ and $\chi_1=1$ outside of $B$. For $\varepsilon>0$, set
\[
\tilde\omega(z):=\omega(z)+\varepsilon\,dd^c\!\left(\chi_1(z)\cdot|z|^2\right).
\]
When $\varepsilon$ is small enough, $\tilde\omega$ is a smooth closed-$(1,1)$ form that is zero on $\overline{V}$ and strictly positive elsewhere. Let $\varphi_2\in C^{\infty}(\mathbb{C}^n)$ be a plurisubharmonic function with $dd^c\varphi_2=\tilde\omega$, and $\chi_2\in C^{\infty}_0(U)$ with $\chi_2=1$ on a neighbourhood of $\overline{V}$ in $U$. For $C>0$, set
\[
	\varphi(z):=\chi_2(z)\cdot\varphi_1(z)+C \varphi_2(z).
\]
The function $\varphi$ is strictly plurisubharmonic on $\mathbb C^n$ if $C$ is large enough, and
\[
\iota_{\Sigma_{k_j}}^*dd^c\varphi=\iota_{\Sigma_{k_j}}^*dd^c\varphi_1=d\!\left(\iota^*_{\Sigma_{k_j}}d^c\varphi_1\right)=0
\]
for each $j$, since the gradient of $\varphi_1$ vanishes on $\Sigma_{k_j}$ for each $j$. This completes the proof that (i) implies (ii).

\smallskip

Note that a similar argument can be used to show the implication (1) $\Longrightarrow$ (3) for totally real sets as discussed in Section~\ref{s.rc}.

\subsection{Proof of Corollary~\ref{c.global}}\label{p.cor}
Suppose that $S\subseteq M$ is a compact rationally convex subset of a regular totally real set $M$. Let $\{\Sigma_{k_j}\}_{j=1,\ldots,r}$ be a regular cover of $S$ satisfying properties (i)-(v) of Proposition \ref{p.BS}. Choose a neighbourhood $U$ of $S$ such that $\{U\cap\Sigma_{k_j}\}_{j=1,\ldots,r}$ is a regular cover of $S$, and $\tilde S:=\overline{U\cap\cup_{j=1}^{r}\Sigma_{k_j}}$ is a totally real set; consequently there exists a strictly plurisubharmonic $\varphi$, defined in a neighbourhood of $\tilde S$, with $\varphi^{-1}(0)=\tilde S$. We can now apply the methods of the proof above, shrinking $U$ as necessary, to extend the domain of $\varphi$ to all of $\mathbb{C}^n$.

Conversely, given the existence of a smooth strictly plurisubharmonic function $\varphi:\mathbb{C}^n\to\mathbb{R}$, a neighbourhood of $U$ of $S$, and a regular cover $\{\Sigma_{k_j}\}_{j=1,\ldots r}$ of $M$ such that $\iota^*_{U\cap\Sigma_{k_j}}dd^c\varphi=0$ for each $j=1,\ldots ,r$, an application of Theorem~\ref{t.BS} shows that, after shrinking $U$ slightly, the closure of $\hat S:=U\cap(\cup_{j=1}^r\Sigma_{k_j})$ is a rationally convex totally real compact. As $S\subset\hat S$, $S$ is rationally convex as well. Indeed, in view of Corollary~\ref{c.approx} any complex-valued continuous function on the closure of $\hat S$ can be approximated uniformly by rational functions. Therefore continuous functions on $S$ can be approxiated uniformly by rational functions, completing the proof~\cite[Theorem 1.2.12]{St}.

\subsection{Proof of (ii) $\Longrightarrow$ (i).}
The proof in this direction is more involved. Assume (ii) holds. In view of Proposition \ref{p.BS}
we may assume that the cover $\{\Sigma_{k_j}\}_{j=1,\ldots, r}$ satisfies properties (i)-(v). Further, note that
the new cover is also isotropic with respect to the given form $dd^c\varphi$, as each element of the new cover is---modulo some holes made by the surgeries---an intersection of members of the original cover. We will show that, for sufficiently small $\delta>0$, the regular cover $\{\Sigma^\delta_{k_j}\}$ has the property that
$\overline\Sigma^\delta=\bigcup_j^r\overline\Sigma^\delta_{k_j}$ is rationally convex. Since $\overline\Sigma^\delta$ shrinks down to $S$ as $\delta\searrow 0$, it will follow that $S$ is rationally convex.

	Let $\rho=\rho_\delta$ denote a strictly plurisubharmonic function in a neighbourhood $U$ of $\overline\Sigma^\delta$ with the property that $\rho^{-1}(\{0\})=\overline\Sigma^\delta$.
	\begin{lemma}[cf. Lemma 3.2 of Duval--Sibony~\cite{DS}]\label{DS.l1}
		For each sufficiently small $\delta>0$ there exists a smooth strictly plurisubharmonic function $\tilde\varphi:\mathbb{C}^n\to\mathbb R$ such that for every $m\in\mathbb N$ there exists a smooth function $h$ in a neighbourhood of $\overline{\Sigma}^\delta$ with the following properties:
	\begin{enumerate}
		\item[(a)]$|h|=\exp(\tilde\varphi+\sigma)$ with $\sigma$ vanishing on $\overline\Sigma^\delta$ and $\sigma\leq-c\cdot\rho $, where $c$ is a positive constant.
		\item[(b)]$\bar\partial h$ vanishes to order $m$ on $\overline\Sigma^\delta$.
\end{enumerate}
\end{lemma}
	We now show that the lemma implies the theorem, following Duval--Sibony~\cite{DS}.

	Set $\overline\Sigma^\delta_\varepsilon:=\{z\in U\,:\,\rho(z)<\varepsilon\}$. Choose $\varepsilon>0$ small enough so that $\overline\Sigma^\delta_{2\varepsilon}$ is contained in $U$ as a relatively compact subset and is pseudoconvex.

	Using H\"ormander's estimates~\cite{Ho}, we solve the equation $\bar\partial u=\bar\partial h$ on $\overline\Sigma^\delta_{2\varepsilon}$ with the estimate $\|u\|^2_{L^2(\overline\Sigma^\delta_{2\varepsilon})}\leq C\|\bar\partial h\|^2_{L^2(\overline\Sigma^\delta_{2\varepsilon})}$, where $C$ can be chosen independently of $\varepsilon$.

	Choose $\eta>0$ small enough so that for each $x\in\overline\Sigma^\delta_\varepsilon$, the ball $B(x,\eta\varepsilon)$ centred at $x$ with radius $\eta\varepsilon$ is contained in $\overline\Sigma^\delta_{2\varepsilon}$.

	For a point $z\in\overline\Sigma^\delta_\varepsilon$, we may apply a lemma of H\"ormander--Wermer~\cite{HW2} to see that
	\begin{align*}
		|u(z)|&\lesssim\varepsilon\sup_{B(z,\eta\varepsilon)}|\bar\partial u|+\varepsilon^{-n}\|u\|_{L^2(B(z,\eta\varepsilon))}\\
		&\lesssim\varepsilon^{m+1}+\varepsilon^{-n}\|\bar\partial h\|_{L^2(\overline\Sigma^\delta_{2\varepsilon})}\\
		&\lesssim\varepsilon^{m+1}+\varepsilon^{-n+m}=\text{O}(\varepsilon^3),
	\end{align*}
	provided $m$ is large.

	Set $\tilde h:=e^{(c/2)\varepsilon}(h-u)$; $\tilde h$ is holomorphic on $\overline\Sigma^\delta_{\varepsilon}$. On $\overline\Sigma^\delta$ we have, for small $\varepsilon>0$,
\[
	|\tilde h|=e^{(c/2)\varepsilon}|h-u|\geq e^{(c/2)\varepsilon}(e^{\tilde\varphi}-\text{O}(\varepsilon^3))=e^{\tilde\varphi+(c/2)\varepsilon}-\text{O}(\varepsilon^3) \geq e^{\tilde\varphi}.
\]
On the other hand, on the boundary $b\overline\Sigma^\delta_\varepsilon$ of $\overline\Sigma^\delta_{\varepsilon}$ we have
\[
	|\tilde h|\leq e^{(c/2)\varepsilon}(e^{\tilde\varphi-c\varepsilon}+\text{O}(\varepsilon^3))\leq e^{\tilde\varphi-(c/2)\varepsilon}+\text{O}(\varepsilon^3)< e^{\tilde\varphi}.
\]
	We can now apply a lemma of Duval--Sibony \cite[Lemma 1.2]{DS}, with $\Omega=\mathbb{C}^n$ (see Section~\ref{s.rc}), to conclude that $\overline\Sigma^\delta$ is rationally convex for sufficiently small $\varepsilon>0$.

\begin{proof}[Proof of Lemma \ref{DS.l1}]
To begin, we use Proposition \ref{p.BS} to construct a new regular cover $\{\Sigma_{k_j}'\}_{j=1,\ldots,r}$ with
$\overline\Sigma_{k_j}'\subset\Sigma_{k_j}$ for each $j$. This cover is also isotropic with respect
to the form $dd^c \varphi$. By part (v) in Proposition \ref{p.BS}, $\overline\Sigma'$ is a retract of some open set $\tilde\Sigma'\subset\mathbb{C}^n$, and by Hatcher~\cite[Prop 1.17]{Ha} there is an injection $H_1(\overline\Sigma',\mathbb{Z})\hookrightarrow H_1(\tilde\Sigma',\mathbb{Z})$. Let $\gamma_1,\ldots,\gamma_p\in H_1(\tilde\Sigma',\mathbb{Z})$ be a basis for the image of this injection, and $\alpha_1,\ldots\alpha_p$ be a corresponding dual basis of closed $1$-forms on $\tilde\Sigma'$.
We may assume that the curves $\gamma_1,\ldots\gamma_p$ are supported on $\overline\Sigma'$; moreover, we may assume that they are piecewise smooth.
\begin{lemma}\label{l.BS}
	For each $\ell=1,\ldots,p$ there exists a smooth compactly supported function $\psi_\ell$ satisfying
\[
	\iota^*_{\overline\Sigma'_{k_j}}d^c\psi_{\ell}=\iota^*_{\overline\Sigma'_{k_j}}\alpha_{\ell},\ \
	j=1,\dots, r.
\]
\end{lemma}

\begin{proof}
	Fix $\ell\in\{1,\ldots,p\}$. We will first construct an open cover $\{U_j\}_{j=1}^{r}$ of $\overline\Sigma'$ with special properties via reverse induction on $j$. Let $U_r$ be a neighbourhood of $\overline\Sigma_{k_r}'$ such that $\overline{U_r\cap\Sigma_{k_r}}\subset\Sigma_{k_r}$. Suppose that $U_r,\ldots,U_{j+1}$ have been constructed. Let $U_{j}$ be a neighbourhood of $\overline\Sigma'_{k_j}\setminus\bigcup_{i> j} U_i$ such that $\overline{U_j\cap\Sigma_{k_j}}\subset\Sigma_{k_j}$ and $U_j\cap\overline\Sigma'_{k_i}=\varnothing$ for any $i>j$. In this way, we construct an open cover $\{U_j\}$ of $\overline\Sigma'$ with the property that for any $j$, $U_j\cap\overline\Sigma'_{k_i}=\varnothing$ whenever $i>j$.

	Now, let $\{\chi_j\}$ be partition of unity subordinate to $\{U_j\}$, and choose $\psi_{\ell,j}\in C^{\infty}_0(U_j)$ such that
\[
	\iota^*_{U_j\cap\Sigma_{k_j}}d^c\psi_\ell=\iota^*_{U_j\cap\Sigma_{k_j}}(\chi_j\alpha_\ell).
\]
	This can be achieved if we assume $\psi_{\ell,j}$ is zero on $U_j\cap\Sigma_{k_j}$ and specify derivatives of $\psi_{\ell,j}$ in the directions contained in $J(T\Sigma_{k_j})$. Set $\psi_\ell:=\sum_{i=1}^r\psi_{\ell,i}$. Then
\[
	\iota^*_{\overline\Sigma'_{k_j}}d^c\psi_\ell=\sum_{i=1}^r\iota^*_{\overline\Sigma'_{k_j}}d^c\psi_{\ell,i}.
\]
	If $\overline\Sigma'_{k_j}\cap U_i=\varnothing$, then clearly $\iota^*_{\overline\Sigma'_{k_j}}d^c\psi_{\ell,i}=0$. If $\overline\Sigma'_{k_j}\cap U_i\neq\varnothing$, then $j\leq i$ by construction, and so $\overline\Sigma'_{k_j}\cap U_i$ is a submanifold of $\Sigma_{k_i}$. In either case,
\[
	\iota^*_{\overline\Sigma'_{k_j}}d^c\psi_{\ell,i}=\iota^*_{\overline\Sigma'_{k_j}}(\chi_i\alpha_{\ell}),
\]
	so we conclude that
\[
	\iota^*_{\overline\Sigma'_{k_j}}d^c\psi_\ell=\sum_{i=1}^r\iota^*_{\overline\Sigma'_{k_j}}(\chi_i\alpha_\ell)=\iota^*_{\overline\Sigma'_{k_j}}\alpha_\ell
\]
	for each $j$.
\end{proof}
Applying the lemma, note that
\[
	\int_{\gamma_s}d^c\psi_{\ell}=\int_{\gamma_s}\alpha_\ell=\delta_{s\ell}
\]
for $1\leq s,\ell\leq p$.

	Set $\varphi_{\lambda}:=\varphi+\lambda_1\psi_1+\ldots+\lambda_p\psi_p$, where $\lambda=(\lambda_1,\ldots,\lambda_p)$ is chosen small enough so that $\varphi_{\lambda}$ is strictly plurisubharmonic on $\mathbb C^n$, and
\[
	\int_{\gamma_\ell}d^c\varphi_\lambda\in 2\pi\mathbb{Z}/ M\hspace{2 em}\text{for each }1\leq\ell\leq p
\]
	and some large integer $M$. Here the assumption that $\iota^*_{\Sigma_{k_j}}dd^c\varphi=0$ for each $j$ has been used.

	Set $\varphi_2:=M\varphi_{\lambda}$ and fix $x_0\in\overline\Sigma'$. It is straightforward to see that the function $g:\overline\Sigma'\to\mathbb{R}/2\pi\mathbb{Z}$, given by
\[
	g(x)=\int_{x_0}^xd^c\varphi_2,
\]
	is well-defined. Indeed, the integration is being taken over some piecewise smooth curve in $\overline\Sigma'$ connecting $x_0$ to $x$, and it is independent of the choice of curve. Define now $h_1:\overline\Sigma'\to\mathbb{C}$ by setting $h_1=e^{\varphi_2}e^{ig}$.

	Shrinking further $\delta>0$ in the statement of Proposition \ref{p.BS} yields another regular cover $\{\Sigma''_{k_j}\}_{j=1}^{r}$ of $S$ with $\overline\Sigma''_{k_j}\subset\Sigma'_{k_j}$.

	We next claim that $h_1$ may be extended to a smooth function $h_2$, defined on a neighbourhood of $\overline\Sigma''=\cup_{j}\overline\Sigma''_{k_j}$ so that $\bar\partial h_2|_{\overline\Sigma''}=0$ and $|h_2|$ agrees with $e^{\varphi_2}$ to order 1 on $\overline\Sigma''$. We will construct the extension locally and patch it together using a partition of unity.

	For $q\in\overline\Sigma''$, let $j(q)=\max\{j\,:\,q\in\Sigma'_{k_j}\}$. Then there exists a neighbourhood $U_q$ so that $\overline{U_q\cap\Sigma'_{k_{j(q)}}}\subset\Sigma'_{k_{j(q)}}$. Write $M_q:=U_q\cap\Sigma'_{k_{j(q)}}$; observe that $M_q$ is a smooth submanifold of $\Sigma_{k_{j(q)}}'$.

		First assume that $\dim M_q=n$. We apply  H\"ormander--Wermer~\cite[Lemma 4.3]{HW2} to extend the function $(\varphi_2+ig)|_{M_q}$ smoothly to a function $\Phi_q$ defined on an open neighbourhood of $M_q$ with the property that $\bar\partial\Phi_r=0$ on $M_q$. (Strictly speaking, $g$ is a multiple-valued function, so we must first choose a local branch of $g$.) We may assume that $\Phi_q$ is defined on $U_q$, by shrinking the neighbourhood if necessary.

	Now assume that $\dim M_q<n$. By shrinking $U_q$ if necessary, we may assume there exists a totally real manifold $\hat M_q$ of maximal dimension $n$ containing $M_q$, and let $N_{q,x}$ be the orthogonal complement of $T_x(M_q)$ in $T_x(\hat M_q)$. Extend the function $(\varphi_2+ig)|_{M_q}$ to a function $\Phi_q$ on $\hat M_q$ with the condition that $d\Phi_q|_{N_q,x}=\alpha|_{N_q,x}$, where $\alpha$ is given by
\[
	\alpha=d\varphi_2+id^c\varphi_2.
\]
	Using the same lemma of H\"ormander--Wermer from the previous paragraph, we may extend $\Phi_q$ further to an open neighbourhood of $\hat M_q$ with $\bar\partial\Phi_q=0$ on $\hat M_q$. For simplicity of notation this extension will also be called $\Phi_q$.

	Let $\{\chi_j\}$ be a partition of unity associated with this covering, along with extensions $\Phi_j$, smooth totally real manifolds $M_j$, and real vector bundles $N_{j,x}$. Set $h_2:=\sum_j\chi_je^{\Phi_j}$. Observe that this function is indeed an extension of $h_1|_{\overline\Sigma''}$ to an open neighbourhood of $\overline\Sigma''$. We have
\[
		\bar\partial h_2=\sum_j\bar\partial\chi_je^{\Phi_j}+\sum_j\chi_je^{\Phi_j}\bar\partial\Phi_j=\sum_j\bar\partial\chi_j(e^{\Phi_j}-\tilde h_1)+\sum_j\chi_je^{\Phi_j}\bar\partial\Phi_j,
\]
where $\tilde h_1$ is any extension of $h_1$ from $\overline\Sigma''$ to $\mathbb{C}^n$. Since the $e^{\Phi_j}$ agree and equal $\tilde h_1$ on $\overline\Sigma''$, we see that $\bar\partial h_2=0$ on $\overline\Sigma''$.

We will now show that $|h_2|=e^{\varphi_2}$ to order 1 on $\overline\Sigma''$. First, we similarly have
	\begin{equation}\label{e.order1}
	dh_2=\sum_je^{\Phi_j}d\chi_j+\sum_j\chi_je^{\Phi_j}d\Phi_j=\sum_jd\chi_j(e^{\Phi_j}-\tilde h_1)+\sum_j\chi_j e^{\Phi_j}d\Phi_j.
	\end{equation}
As before, the first sum on the right side of \eqref{e.order1} above vanishes on $\overline\Sigma''$. Fix $x\in\overline\Sigma''$ and $v\in T_x\mathbb{C}^{n}\cong T_x\mathbb{R}^{2n}$. Note that
\[
T_x\mathbb{C}^{n}=T_xM_j\oplus J(T_xM_j)\oplus N_{j,x}\oplus J(N_{j,x})
\]
for every $j$ with $x\in M_j$. So if $v\in T_x M_j$, then
\[
d\Phi_j(v)=dh(v)=d(\varphi_2+ig)(v)=d\varphi_2(v)+id^c\varphi_2(v)=\alpha(v).
\]
If $v\in J(T_xM_j)$, then applying the above expression yields
\[
id\Phi_j(v)=d\Phi_j(J(v))=\alpha(J(v))=i\alpha(v).
\]
If $v\in N_{j,x}$, then by construction,
\[
d\Phi_j(v)=\alpha(v),
\]
and similarly,
\[
id\Phi_j(v)=d\Phi_j(J(v))=\alpha(J(v))=i\alpha(v)
\]
whenever $v\in J(N_{j,x})$. By linearity, we see that $d\Phi_j=\alpha$ at the point $x$. Consequently, we conclude through linearity that $d\Phi_j=\alpha$ on $\overline\Sigma''\!$, as $x$ was chosen arbitrarily.

	Applying this to \eqref{e.order1} shows that on $\overline\Sigma''$ we have
	\begin{equation}\label{e.log1}
	dh_2=h_2\alpha=h_2(d\varphi_2+id^c\varphi_2) .
	\end{equation}

	Define a holomorphic branch of the logarithm, $L$, near $h_2(x)$. Since $\bar\partial h_2=0$ on $\overline\Sigma''\!$,
	\begin{equation}\label{e.log2}
	d(L(h_2))=d(\log |h_2|+i\arg(h_2))=d(\log |h_2|)+id(\arg(h_2));
	\end{equation}
on the other hand, applying \eqref{e.log1} shows
	\begin{equation}\label{e.log3}
	d(L(h_2))=\partial(L(h_2))+\bar\partial(L(h_2))=\partial(L(h_2))=\frac{\partial h_2}{h_2}=\frac{dh_2}{h_2}=d\varphi_2+id^c\varphi_2.
	\end{equation}
Comparing the real parts of \eqref{e.log2} and \eqref{e.log3} yields
\[
d(\log |h_1|)=d\varphi_2 ,
\]
and hence $|h_1|=e^{\varphi_2}$ to order 1 at points of $\overline\Sigma''$.

Shrinking $\delta>0$ even further gives a regular cover $\{\Sigma'''_{k_j}\}_{j=1,\ldots, r}$ of $S$ with the property that $\overline\Sigma'''_{k_j}\subset\Sigma''_{k_j}$ for each $j$.
\begin{lemma}\label{l2.BS}
	The function $h_2$ can be further modified to a function $h$ on a neighbourhood of $\overline\Sigma'''\!\!$ with the additional property that $\bar\partial h =0$ to order $m$ on $\overline\Sigma'''\!:=\cup_j\overline\Sigma'''_{k_j}$.
\end{lemma}
\begin{proof}
	As before, for each $q\in\overline\Sigma'''\!$ we may write $j'(q)=\max\{j\,:\,q\in\Sigma''_{k_j}\}$; there exists a neighbourhood $U_q$ of $q$ so that $\overline{U_q\cap\Sigma''_{k_{j'(q)}}}\subset\Sigma''_{k_{j'(q)}}$ and $U_q\cap\overline\Sigma_{k_i}=\varnothing$ for $i>j'(q)$. Set $M'_q:=U_q\cap\Sigma''_{k_{j'(q)}}$. Following the proof of H\"ormander--Wermer~\cite[Lemma 4.3]{HW2}, we find that, after possibly shrinking $U_q$, we can construct a local extension $\hat h_q$ of $h_2|_{\overline\Sigma'''}$ on $U_q$ with $\bar\partial\hat h_q$ vanishing to order $m$ on $\overline\Sigma'''$ and such that
\[
\hat h_q-h_2=\text{O}\big(\text{dist}(\,\cdot\,,M_q)^m).
\]
	Now let $\{U_j\}$ be an associated finite covering of $\overline\Sigma'''$ with associated extensions $\hat h_j$, and let $\{\chi_j\}$ be a partition of unity subordinate to the cover. Set $h:=\sum_j\chi_j\hat h_j$. Then $h$ is equal to $h_2$ on $\overline\Sigma'''$. For a fixed $x\in\overline\Sigma'''$, we have
	\begin{align*}
		\bar\partial h&=\sum_j\chi_j\bar\partial \hat h_j+\sum_j\hat h_j\bar\partial\chi_j=\sum_j\chi_j\bar\partial \hat h_j+\sum_j(\hat h_j-h_2)\bar\partial\chi_j\\
		&=\sum_j\chi_j\cdot\text{O}\big(\text{dist}(\,\cdot\,,M_j)^m\big)+\sum_j\text{O}\big(\text{dist}(\,\cdot\,,M_j)^m\big)\cdot\bar\partial\chi_j\\
		&\leq\text{O}\big(\text{dist}(\,\cdot\,,\overline\Sigma''')^m\big)
	\end{align*}
near $x$, because the open cover was constructed so that $U_j\cap\overline\Sigma'''\!\!\subset M_j$ for each $j$. We conclude $\bar\partial h$ vanishes to order $m$ at points of $\overline\Sigma'''\!$.
\end{proof}

To show part (a) of Lemma \ref{DS.l1}, we repeat the proof of Lemma 3.3 in Duval--Sibony~\cite{DS}. Because $\overline\Sigma'''\!$ is totally real, there exists a nonnegative strictly plurisubharmonic function $\rho$ in a neighbourhood of $\overline\Sigma'''\!$ with $V\cap\{x\,:\,\rho(x)=0\}=\overline\Sigma'''\!$. Notice that for $\varepsilon>0$ and $\tau>0$ there exists a $f\in C^{\infty}([0,\infty))$ supported on $[0,\tau]$ such that $f(t)=t$ for $t$ small and $f'(t)\geq-\varepsilon$, $tf''(t)\geq-\varepsilon$ for every $t$. (This is the same function used in the proof of Lemma 3.3 in Duval--Sibony~\cite{DS}.) Set $\theta:=\varphi_2-\log|h_2|$; note that $\theta$ is strictly plurisubharmonic on $\overline\Sigma'''$.

Choose $A>0$ such that
\begin{equation}\label{e.ineq}
\theta\geq-(A/2)\rho
\end{equation}
in some neighbourhood of $\overline\Sigma'''\!$. We also choose $\tau>0$ small enough so that $\rho$ is strictly plurisubharmonic on $\{x\,:\,\rho(x)\leq\tau\}$; therefore we have a neighbourhood $V$ of $\overline\Sigma'''\!\!$ on which $\theta$ and $\rho$ are strictly plurisubharmonic and on which \eqref{e.ineq} holds. Fix $\varepsilon>0$ such that on $V$ we have
\[
3\varepsilon A\rho^{-1}d\rho\wedge d^c\rho\leq dd^c\theta\hspace{2em}\text{and}\hspace{2em}3\varepsilon A dd^c\rho\leq dd^c\theta .
\]
Indeed, this is possible as $\theta$ is strictly plurisubharmonic on $V$ and all forms which are being compared are (1,1) forms; also, note that multiplication by $\rho^{-1}$ does not introduce singularities to the form $d\rho\wedge d^c\rho$ since $\rho$ has no first order terms in its Taylor expansion at points of $\overline\Sigma'''$. Now,
\[
dd^c\big(\theta+Af(\rho)\big)=Af''(\rho)d\rho\wedge d^c\rho+Af'(\rho)dd^c\rho+dd^c\theta\geq\frac13dd^c\theta>0
\]
on $V$. Moreover, near $\overline\Sigma'''\!\!$ we have
\[
\theta+Af(\rho)\geq-\frac{A}{2}\rho+A\rho\geq\frac{A}{2}\rho.
\]
Setting $\tilde\varphi:=\varphi_2+Af(\rho)$ completes the proof of the lemma and therefore the proof that (ii) implies (i).
\end{proof}


\end{document}